\documentclass[11pt]{article}
\usepackage{latexsym, amsmath, amssymb,amsthm,amsopn,amsfonts}
\usepackage{version}
\usepackage{epsfig,graphics,color,graphicx,graphpap}
\usepackage{graphicx, caption, subcaption}
\usepackage{amssymb}
\usepackage{todonotes}
\usepackage{listings}
\usepackage{mathrsfs}
\usepackage[citecolor=blue]{hyperref}

\usepackage[framemethod=tikz]{mdframed}
\usepackage{tikz}
\usepackage{url}
\definecolor{mycolor}{rgb}{0.122, 0.435, 0.698}

\newmdenv[innerlinewidth=0.5pt, roundcorner=4pt,linecolor=mycolor,innerleftmargin=6pt,
innerrightmargin=6pt,innertopmargin=6pt,innerbottommargin=6pt]{mybox}
\usepackage{url}
\definecolor{mycolor}{rgb}{0.122, 0.435, 0.698}

\newcommand{\redsout}{\bgroup\markoverwith{\textcolor{red}{\rule[0.5ex]{2pt}{.4pt}}}\ULon}
\newcommand{\LV}{\left|}
\newcommand{\RV}{\right|}
\newcommand{\LC}{\left(}
\newcommand{\RC}{\right)}

\newcommand{\LA}{\left<}
\newcommand{\RA}{\right>}
\newcommand{\p}{\partial}

\newenvironment{keywords}{\noindent{\bf Key words.}\small}{\par\vspace{1ex}}
\newenvironment{AMS}{\noindent{\bf AMS subject classifications 2010.}\small}{\par}

\numberwithin{equation}{section}
\newtheorem{theorem}{Theorem}[section]
\newtheorem{corollary}{Corollary}[section]
\newtheorem{proposition}{Proposition}[section]
\newtheorem{lemma}{Lemma}[section]

\newtheorem{remark}{Remark}[section]
 
\newcommand{\R}{\mathbb R}
\title{Single pixel X-ray Transform and Related Inverse Problems}
\author{
	Ru-Yu Lai\thanks{School of Mathematics, University of Minnesota, Minneapolis, MN 55455, USA; 	\href{mailto:rylai@umn.edu}{rylai@umn.edu}
	}
	\and
	Gunther Uhlmann\thanks{Department of Mathematics, University of Washington, Seattle, WA 98195, USA; HKUST Jockey Club Institute for Advanced Study, HKUST, Clear Water Bay, Kowloon, Hong Kong;
		\href{mailto:gunther@math.washington.edu}{gunther@math.washington.edu}
	}
	\and
	       Jian Zhai\thanks{School of Mathematical Sciences, Fudan University, Shanghai 200433, China;
		\href{mailto:jianzhai@fudan.edu.cn}{jianzhai@fudan.edu.cn}
		}	
	\and
	Hanming Zhou\thanks{Department of Mathematics, University of California Santa Barbara, Santa Barbara, CA 93106, USA; 
		\href{mailto:hzhou@math.ucsb.edu}{hzhou@math.ucsb.edu}
		}
}
 
\begin{document}
\maketitle

\date{}
 
\begin{abstract}
In this paper, we analyze the nonlinear single pixel X-ray transform $K$ and study the reconstruction of $f$ from the measurement $Kf$. Different from the well-known X-ray transform, the transform $K$ is a nonlinear operator and uses a single detector that integrates all rays in the space. We derive stability estimates and an inversion formula of $K$.  
We also consider the case where we integrate along geodesics of a Riemannian metric.
Moreover, we conduct several numerical experiments to corroborate the theoretical results.

\vskip2cm
\end{abstract}
\begin{keywords}
	 X-ray transform, single pixel X-ray transform, inverse problems
\end{keywords}
\begin{AMS}
	 35R30.
\end{AMS}

\section{Introduction}
In this paper, we study the \textit{single pixel X-ray transform} $K$ defined by 
$$
Kf(x) :=\int_{\mathbb{S}^{n-1}} e^{-Xf(x,\theta)} \,d\theta,  
$$
whose exterior integral integrates over the entire unit sphere $\mathbb{S}^{n-1}$ in $\R^n$, $n\geq2$. The inner function consists of an exponential function and the conventional X-ray transform $X$ defined by
\begin{align}\label{DEF:Xray}
	Xf(x,\theta) := \int_\ell f\,ds=\int_\R f(x+s\theta) \,ds ,\qquad (x,\theta)\in\R^n\times\mathbb{S}^{n-1},
\end{align}
where $\ell=x+s\theta$ is a line passing through a point $x\in\R^n$, in the direction $\theta\in\mathbb{S}^{n-1}$.

The standard X-ray transform consists of recovering a function supported in a bounded domain from its integrals along straight lines through this domain. In dimension two ($n=2$), it coincides with the Radon transform \cite{Radon1917}, which provides the theoretical underpinning for several medical imaging techniques such as Computed Tomography (CT) and Positron Emission Tomography (PET). The X-ray transform has been extensively studied, including its uniqueness, stability estimates and reconstruction formula, see for example, the book \cite{Natterer2001book}. Generalizations of the standard X-ray transform include integrals of tensor fields or along curved lines. We refer to recent survey papers \cite{IM2019, PSU2014} and the references therein for more details.

A notable difference between the X-ray transform $X$ and the single pixel X-ray transform $K$ is the nonlinearity due to the exponential function.
We will discuss later how this nonlinearity of $K$ plays a crucial role in practical applications and why it introduces difficulties in the reconstruction.

The objective of this paper is to recover $f$ from the data $Kf$ by establishing a reconstruction formula and deriving stability estimates. The global uniqueness of the inverse problem was proved by the second author in \cite{KDMV16} relying on the monotonic property of $K$ and the known uniqueness of X-ray transform $X$ \cite{Helgason, Natterer2001book}. 
Due to the special structure of the transform $K$, however, it is not clear that if the same technique in \cite{KDMV16} can be directly applied to the study of stability estimates and reconstruction formulas.

\subsection{Motivation}
The single pixel X-ray transform finds applications in protection of information in highly sensitive systems. The nonlinearity of the transform $K$ acts as a shield against the disclosure of such information. Here the exponential function is chosen to be the nonlinear function in $K$ since attenuation is naturally exponential in space. Specifically, this nonlinearity ensures that there is no one-to-one correspondence between the density $f$ and the true mass $\int_{\mathbb{S}^{n-1}}\int_\R f(x+s\theta)dsd\theta$ and, therefore, $f$ cannot be estimated from a single projection. This thus protects the detailed information of the system such as its structure and composition while the global uniqueness result can still provide a theoretical validation of the system's authenticity.

\subsection{Main results}
The transform $K$ is nonlinear and a monotone decreasing map due to the exponential function in the definition. The nonlinearity of the transform helps secure information; on the other hand, it also introduces difficulties to the mathematical and practical reconstruction of $f$.

As mentioned earlier, the global uniqueness of $K$ was proved in \cite{KDMV16}. However, the inversion formula of $K$, to the best of the authors' knowledge, have not been derived and stability estimate has not yet been investigated.  
The first result we study here is to establish a reconstruction formula of $K$ provided that $f=\varepsilon g$ is sufficiently small for a real-valued number $\varepsilon>0$. 
	
The monotonicity of $K$ implies that when $f$ is increasing, the measurement $Kf$ becomes decreasing and could be eventually very small, which makes it challenging to distinguish the true measurement from noise if the noise does exist. Hence we consider small density $f$ so that the measurement $Kf$ will not be too small in this setting.

Let $\Omega$ be an open bounded domain in $\mathbb R^n$, $n> 2$. We define the space  $$\mathcal{M}_k:=\{f\in C^k(\R^n):\, f\hbox{ is supported in $\overline\Omega$}\}.$$

\begin{theorem}[Inversion formula]\label{thm:formula}
     Let $\Omega$ be an open bounded domain in $\mathbb{R}^n$, $n\geq 2$ with smooth boundary.
	For any $g\in\mathcal{M}_0$, assume that we know $K[\varepsilon g]$ for all $\varepsilon\in \mathbb R$ sufficiently close to zero, then
	$$
	g= - c_n|D| (\p_\varepsilon|_{\varepsilon=0} K[\varepsilon g]),
	$$
	where $c_n=(2\pi|\mathbb{S}^{n-2}|)^{-1}$ and $D=(-\Delta)^{1/2}$ is the square root of Laplacian, which is a pseudo-differential operator. %a nonlocal operator.
\end{theorem}

Theorem~\ref{thm:formula} states that a function $g$ can be reconstructed through this formula based on the linearized data. It also immediately implies the stability estimate for $g$, see Corollary~\ref{cor:stability}.
The proof of Corollary~\ref{cor:stability} follows directly from Proposition~\ref{prop:stability} and the inversion formula of $K$.

\begin{corollary}[Stability estimate]\label{cor:stability}
	Let $\Omega$ be an open bounded domain in $\mathbb{R}^n$, $n>2$ with smooth boundary and let $\Omega_1$ be a larger open and bounded domain satisfying $\overline{\Omega}\subset \Omega_1$.
	There exists a constant $C>0$ depending on $n,\Omega,\Omega_1$ so that 
	\begin{align}\label{EST:g}
		C^{-1}\|g\|_{L^2(\R^n)}\leq  \|\p_\varepsilon|_{\varepsilon=0} K[\varepsilon g]\|_{H^{1}(\Omega_1)}   
	\end{align} 
	 for $g\in \mathcal{M}_0$.  
\end{corollary}

Moreover, we also study the following stability estimates.
%%%%%%%%%%%%%%%%%%%%%%%%%%%%
\begin{theorem}\label{thm:stability} Let $\Omega\subset\R^n$, $n>2$ be an open and bounded domain and let $\Omega_1$ be a larger open and bounded domain so that $\overline{\Omega}\subset\Omega_1$. 
Then for any $f_1,f_2\in \mathcal{M}_1$ with
$\|f_j\|_{C^1(\R^n)}\leq M$, $j=1,2$, for some fixed positive constant $M$,
we have stability estimates
\begin{align*}
   \|Kf_1- Kf_2\|_{H^1(\Omega_1)} \leq C(1+ e^{CM}+Me^{CM})\|f_1-f_2\|_{L^2(\R^n)} +C(e^{CM}-1) \| \nabla (f_1-f_2) \|_{L^2(\R^n)},
\end{align*} 
where the positive constant $C$ is independent of $M$. 

Moreover, assume that $f_1, f_2$ satisfy 
\begin{align}\label{inverse poincare}
\|f_1-f_2\|_{L^2(\R^n)} \geq c \|\nabla (f_1-f_2)\|_{L^2(\R^n)}
\end{align}
for a fixed constant $c>0$, if $M$ is sufficiently small, then we have
\begin{align}\label{EST:KOmega1}
\widetilde{C}\| f_1-f_2\|_{L^2(\R^n)}\leq \|Kf_1-Kf_2\|_{H^1(\Omega_1)},
\end{align}
where the positive constant $\widetilde{C}$ depends on $c,M$. %with $M$ sufficiently small.
\end{theorem}

In Theorem~\ref{thm:stability}, the estimate \eqref{EST:KOmega1} implies that the data $Kf$ in $\Omega_1$ is sufficient to give local stability estimate. Under the assumption \eqref{inverse poincare}, the left hand side of \eqref{EST:KOmega1} can be replaced by $\|f_1-f_2\|_{H^1(\mathbb R^n)}$ with another constant $\widetilde C$.

We apply the linearization scheme to investigate this nonlinear inverse problem, namely, reconstructing $f$ from the measurement $Kf$. Indeed, to study nonlinear inverse problems, it is classical to utilize the linearization scheme and then reduce it to the problem of their linearization, where the existing results, such as injectivity, are utilized to identify the unknown property \cite{Isakov93}. We would also like to note that a general result is proved in \cite{stefanov2009linearizing} when linearizing nonlinear inverse problems. It gives H\"older type estimates for the nonlinear problem under some conditions.
In Section~\ref{Sec:IP}, we linearize the transform $K$ around zero function. This is motivated by the following observation. Since the first nonconstant term of Taylor's expansion of $Kf$ is the normal operator of the X-ray transform $X$, linearizing $Kf$ then reveals this term while the remaining higher order terms vanish. Additionally, thanks to the previously established results for the X-ray transform, we can derive a local reconstruction formula of $K$ and also stability estimates for small enough $f$. For a general function $f$ (not necessary small), however, it would be more challenging to stably recover $f$ since the higher order terms dominate the behavior of $f$. We do not consider this issue here.

In this paper, we also study the single pixel X-ray transform $K$ in the Riemannian case. We establish the uniqueness of $K$ on compact manifolds with boundary, on which the geodesic X-ray transform $X$ is injective. Our proof is a generalization of the argument for the Euclidean case \cite{KDMV16}, by applying the Santalo's formula.

Besides the above theoretical results, we conduct numerical reconstructions for the single pixel X-ray transform $K$ by an optimization method. These experiments provide numerical evidence for our stability estimates of $K$. In particular, if the magnitude of $f$ is small, the reconstruction of $f$ from $Kf$ works quite well, even in the presence of mild noises. While in the case of large $f$, the optimization approach could fail, which suggests that the estimate \eqref{EST:KOmega1} might not hold when $M$ is large.

\section{Inverse problems}\label{Sec:IP}
\subsection{Preliminary results}
We introduce several known results for the X-ray transform. For a function $f\in L^1(\R^n)$, the X-ray transform  
$$
    Xf(x,\theta) = \int_\R f(x+s\theta)\,ds 
$$
is well-defined and is constant along lines in direction $\theta$, that is, $Xf(x,\theta) = Xf(x+t\theta,\theta)$ and, moreover, $Xf(x,\theta)=Xf(x,-\theta)$ for $t\in\R$, $x\in\R^n$ and $\theta\in\mathbb{S}^{n-1}$.
Let $\Sigma:=\{(z,\theta)\in \mathbb R^n\times\mathbb S^{n-1}: z\in \theta^{\perp}\}$, where $\theta^\perp:=\{z\in\mathbb R^n: z\perp \theta\}$ is the orthogonal complement of $\theta$, be the parameter set of straight lines, then $$X: C^\infty_c(\mathbb R^n)\to C^\infty_c(\Sigma)$$
is a continuous map. 
In particular, $Xf$ is compactly supported in $\Sigma$ if $f$ is compactly supported. We denote the adjoint of $X$, under the $L^2$ inner product, by $X'$ and the normal operator by $X'X$. Then $X': C^\infty_c(\Sigma)\to C^\infty(\mathbb R^n)$ has the expression
$$X' \psi (x) :=\int_{\mathbb S^{n-1}} \psi(x-(x\cdot \theta) \theta, \theta)\, d\theta.$$

Then we have the following results, see \cite{Natterer2001book, StefanovUhlmannBook} for detailed discussions and proofs. 
\begin{lemma}\label{lemma:normal} %(\cite{StefanovUhlmannBook})
For $f\in\mathcal{S}(\R^n)$ (Schwartz space),
	\begin{align*} 
	X'Xf(x)  
 	 = \int_{\mathbb{S}^{n-1}} \int_\R f(x+s\theta)ds \,d\theta 
	 =2\int_{\R^n} {f(y)\over |x-y|^{n-1}} \, dy.
	\end{align*} 
\end{lemma} 

An inversion formula of the X-ray transform is stated below.
\begin{proposition}\label{prop:formula} %[\cite{StefanovUhlmannBook}] 
	For $f\in\mathcal{S}(\R^n)$, 
	$$
	f= c_n|D|X'Xf, 
	$$
	where $c_n=(2\pi|\mathbb{S}^{n-2}|)^{-1}$, $D=(-\Delta)^{1/2}$ is a non-local pseudo-differential operator and $|\mathbb{S}^{n-2}|$ is the measure of the unit sphere $\mathbb{S}^{n-2}$.
\end{proposition}
This proposition implies that, up to a positive constant, $X'X$ is the Fourier multiplier $|\xi|^{-1}$.
Note that the inversion formula of Proposition~\ref{prop:formula} remains true for any distribution $f$ with compact support.

Throughout this paper, we use $C$ to denote positive constants, which may change from line to line.
 
Next proposition shows the stability of the inversion. 
\begin{proposition}\label{prop:stability} %[\cite{StefanovUhlmannBook}] 
	Let $\Omega$ be an open bounded domain in $\mathbb{R}^n$, $n>2$ with smooth boundary and let $\Omega_1$ be a larger open and bounded domain satisfying $\overline{\Omega}\subset \Omega_1$. For any nonnegative integer $s$, there is a constant $C>0$ so that 
	\begin{align}\label{normal stability} 
    C^{-1}\|f\|_{H^s(\R^n)}\leq \|X'Xf\|_{H^{s+1}(\Omega_1)}\leq C\|f\|_{H^s(\R^n)}
	\end{align} 
	for $f\in H^s(\R^n)$ supported in $\overline\Omega$.
\end{proposition}

Similar stability estimates in the Riemannian setting have been established in \cite{SU2004}. Notice that the estimate \eqref{normal stability} is associated with the normal operator $X'X$, which has strong connections to the operator $K$. There also exist stability results regarding the transform $X$ itself, see e.g. \cite[Section II.5]{Natterer2001book}.
Since Proposition \ref{prop:stability} is crucial in the derivation of our main results later, we provide the proof here following \cite{StefanovUhlmannBook}. 

\begin{proof}[Proof of Proposition~\ref{prop:stability}]
The second inequality of \eqref{normal stability} follows from the fact that $X'X$ is the Fourier multiplier $c_n^{-1}|\xi|^{-1}$ by Proposition~\ref{prop:formula}. To this end, we first estimate
\begin{align*}
    \|X'Xf\|^2_{H^{s+1}(\mathbb R^n)} & =\|(1+|\xi|^2)^{(s+1)/2}\mathcal F(X'Xf)(\xi)\|^2_{L^2(\mathbb R^n)}=\|c_n^{-1}(1+|\xi|^2)^{(s+1)/2}|\xi|^{-1} \widehat f(\xi)\|^2_{L^2(\mathbb R^n)}\\
    & \leq C \LC\|(1+|\xi|^2)^{s/2}\widehat f(\xi)\|^2_{L^2(|\xi|> 1)}+\||\xi|^{-1}\widehat f(\xi)\|^2_{L^2(|\xi|\leq 1)}\RC.
\end{align*}
Here we denote the Fourier transform of a function $f$ by $\hat{f}$ or $\mathcal{F}(f)$.
It remains to estimate the second term $\||\xi|^{-1}\widehat f(\xi)\|^2_{L^2(|\xi|\leq 1)}$. We have
$$
    |\widehat{f}(\xi)|=\LV \int e^{-ix\cdot\xi} f(x)\,dx\RV\leq \|f\|_{H^s(\R^n)} \|\phi_\xi \|_{H^{-s}(\R^n)},
$$
where $\phi_\xi(x):= e^{-ix\cdot\xi}\chi_\Omega(x)$ with $\chi_\Omega\in C^\infty_0(\R^n)$ equals $1$ in a neighborhood of  $\overline\Omega$. Then
\begin{align*}
    \||\xi|^{-1}\widehat f(\xi)\|^2_{L^2(|\xi|\leq 1)} &= \int_{|\xi|\leq 1} |\xi|^{-2} |\widehat{f}(\xi)|^2 \,d\xi\\
    & \leq \LC\int_{|\xi|\leq 1} |\xi|^{-2}\,d\xi\RC \|f\|^2_{H^s(\R^n)} \max_{|\xi|\leq 1}\|\phi_\xi \|^2_{H^{-s}(\R^n)}.
\end{align*}
Note that $\|\phi_\xi \|_{H^{-s}(\R^n)}\leq C$, where the constant $C>0$ depends on $n,s$ for each $|\xi|\leq 1$. Moreover, $\int_{|\xi|\leq 1} |\xi|^{-2}\,d\xi<\infty$ for $n\geq 3$. Hence, $\||\xi|^{-1}\widehat f(\xi)\|^2_{L^2(|\xi|\leq 1)}\leq C \|f\|^2_{H^s(\R^n)}$ also holds. This proves the second inequality of \eqref{normal stability} by observing that
$$\|X'Xf\|_{H^{s+1}(\Omega_1)}\leq \|X'Xf\|_{H^{s+1}(\mathbb R^n)}\leq C\|f\|_{H^s(\mathbb R^n)}.$$

To show the first inequality of \eqref{normal stability}, we begin by applying Proposition~\ref{prop:formula} to get that
\begin{align}\label{first inequality}
    \|f\|^2_{H^s(\mathbb R^n)}\leq C\|X'X f\|^2_{H^{s+1}(\mathbb R^n)}=C\bigg(\|X'Xf\|^2_{H^{s+1}(\Omega_1)}+\|X'Xf\|^2_{H^{s+1}(\mathbb R^n\setminus \Omega_1)}\bigg).
\end{align}
The operator $X'X: H^s_0(\Omega)\to H^{s+1}(\mathbb R^n\setminus \Omega_1)$ (note that $f$ is supported in $\overline\Omega$) has a smooth kernel, so it is compact. Together with the fact that $X'X: H^s_0(\Omega)\to H^{s+1}(\Omega_1)$ is injective, one can remove the term $\|X'Xf\|_{H^{s+1}(\mathbb R^n\setminus\Omega_1)}$ from the estimate \eqref{first inequality}, see \cite{StefanovUhlmannBook} and \cite[Proposition V 3.1]{Taylor1981}. This proves the first inequality of \eqref{normal stability}.
\end{proof}

To conclude this section, we establish the following mapping properties of the operator $K$.
 
\begin{proposition}
	Let $\Omega$ and $\Omega_1$ be two open bounded domains in $\mathbb R^n$ with $\Omega\subset \overline\Omega\subset \Omega_1$. Then there exists a positive constant $C$, depending on $n$, $\Omega$ and $\Omega_1$, such that
	$$\|Kf\|_{H^1(\Omega_1)}\leq Ce^{C\|f\|_{C^0(\R^n)}} (1+\|f\|_{H^1(\mathbb R^n)})$$
	for any $f\in C^1(\mathbb R^n)$ supported in $\overline\Omega$.
\end{proposition}
This implies that $Kf$ is well-defined in $H^1$ norm.
Notice that $K[0]=|\mathbb{S}^{n-1}|$ and, therefore, $Kf\neq 0$ even if $f\equiv 0$.
Due to this fact, it is worth emphasizing that $Kf$ in general is not in $L^2(\R^n)$.   
\begin{proof}
	Since $f\in C^1(\R^n)$ with support in $\overline\Omega$, one can denote $M:=\|f\|_{C^0(\R^n)}$ for some finite constant $M\geq 0$, which implies that $|e^{-Xf}|$ is bounded by $e^{CM}$. Then
	\begin{align*}
		\|Kf\|^2_{L^2(\Omega_1)} & = \int_{\Omega_1}\LV\int_{\mathbb S^{n-1}} e^{-Xf(x,\theta)}\, d\theta \RV^2\, dx\leq \int_{\Omega_1}\LV \int_{\mathbb S^{n-1}} e^{CM}\, d\theta\RV^2\, dx\leq Ce^{2CM},
	\end{align*}
	where the constant $C$ depends on $n,\Omega,\Omega_1$.
	Moreover, we have
	\begin{align*}
		\|\nabla Kf\|^2_{L^2(\Omega_1)} & =\int_{\Omega_1} \LV\nabla_x\int_{\mathbb S^{n-1}} e^{-Xf(x,\theta)}\, d\theta\RV^2\, dx=\int_{\Omega_1}\LV \int_{\mathbb S^{n-1}} -e^{-Xf} \LC\nabla_x\int f(x+s\theta)\,ds\RC d\theta\RV^2\, dx\\
		& \leq e^{2CM} \int_{\Omega_1} \LC \int_{\mathbb S^{n-1}}\int \LV\nabla_x f(x+s\theta)\RV \, ds d\theta \RC^2\, dx
		=e^{2CM}\int_{\Omega_1} (X' X |\nabla f|)^2\,dx\\
		& = e^{2CM}\| X'X (|\nabla f|) \|^2_{L^2(\Omega_1)}\leq Ce^{2CM} \|\nabla f\|^2_{L^2(\mathbb R^n)}\leq C e^{2CM}\|f\|^2_{H^1(\mathbb R^n)},
	\end{align*}    
	where we applied Proposition~\ref{prop:stability} with $s=0$ to derive the second last inequality. Combining the two estimates for $Kf$ together yields the result.
\end{proof}

\subsection{A reconstruction formula}\label{sec:formula} 
To study the inverse problem,  
we replace $e^{-Xf(x,\theta)}$ in $Kf$ by its Taylor expansion and then obtain
\begin{align}\label{Kf taylor}
	Kf(x)   
	&= \int_{\mathbb{S}^{n-1}} \LC 1- Xf(x,\theta) + Rf(x,\theta)  \RC \,d\theta \notag\\
	&= |\mathbb{S}^{n-1}|- X'Xf(x) + \int_{\mathbb{S}^{n-1}} Rf(x,\theta) \,d\theta,
\end{align} 
where the higher order terms are denoted by
\begin{align*}%\label{DEF:R}
R f(x,\theta) :=  \sum^\infty_{m=2} {(-1)^m\over m!} (Xf)^m(x,\theta)
\end{align*}
and then $\int_{\mathbb{S}^{n-1}}R f(x,\theta)\,d\theta$ is finite for $f\in \mathcal{M}_0$.

We linearize the transform $K$ around the zero function so that the problem is reduced to the inverse problem for the X-ray transform.
\begin{proof}[Proof of Theorem~\ref{thm:formula}]
	Now we take $f=\varepsilon g\in \mathcal{M}_0$ and let $\varepsilon>0$ be a sufficiently small real number. We differentiate $K[\varepsilon g]$ with respect to $\varepsilon$ at $\varepsilon=0$, denoted by $\p_\varepsilon|_{\varepsilon=0} K[\varepsilon g]$, and then obtain
	\begin{align}\label{first linearization}
		\p_\varepsilon|_{\varepsilon=0} K[\varepsilon g](x)&=\lim\limits_{\varepsilon\rightarrow 0} \varepsilon^{-1} (K[\varepsilon g](x) - K[0](x))  \notag\\
		&= - X'Xg(x) +  \lim\limits_{\varepsilon\rightarrow 0}  \int_{\mathbb{S}^{n-1}} \sum^\infty_{m=2} \varepsilon^{m-1} {(-1)^m\over m!}(Xg)^m(x,\theta)  \,d\theta \notag\\
		&= - X'Xg(x),
	\end{align} 
	where we used the fact that $X$ and $X'X$ are linear operators and also $K[0]=|\mathbb{S}^{n-1}|$.
	
	Applying non-local operator $c_n|D|$ to both sides of \eqref{first linearization}, the reconstruction formula of X-ray transform in Proposition~\ref{prop:formula} yields the following reconstruction formula of $K$:
	\begin{align}\label{first reconstruction formula}
	c_n|D| (\p_\varepsilon|_{\varepsilon=0} K[\varepsilon g]) = -c_n|D| X'Xg = -g.
	\end{align}
	Hence the proof of Theorem~\ref{thm:formula} is complete.
\end{proof}

\begin{remark}
By Proposition~\ref{prop:stability} with $s=0$, the reconstruction formula \eqref{first reconstruction formula} also leads to the stability estimate of $g\in \mathcal{M}_0$ in Corollary~\ref{cor:stability}. Since we take the limit $\varepsilon \rightarrow 0$ in the linearization process, this suggests that infinitely many measurements are needed to recover $g$.
\end{remark}

\subsection{Stability estimate}
We are ready to show that the reconstruction of $f$ from the data $Kf$ is stable under suitable assumptions.
\begin{proof}[Proof of Theorem~\ref{thm:stability}] 
Suppose that $\|f_j\|_{C^1(\R^n)}\leq M$ for some constant $M>0$. For any $x\in\Omega_1$, from \eqref{Kf taylor}, we have
\begin{align}\label{ID:f1f2}
    \hskip.5cm K f_1(x)-Kf_2(x) 
    = (X'Xf_2 - X'Xf_1 )(x) + X' (R f_1-R f_2 )(x).
\end{align} 	
By direct computations, the remainder term $X' (R f_1-R f_2 )(x)$ satisfies
\begin{align*}
	|X' (R f_1-R f_2 )(x)|&\leq \LC CM+ {(CM)^2\over 2!}+{(CM)^3\over 3!}+\cdots \RC(X'X|f_1 -f_2| )(x)\\
	&\leq (e^{CM}-1) (X'X|f_1 - f_2| )(x),
\end{align*}
which yields the following two estimates:
\begin{align*}
    |K f_1(x)-Kf_2(x)|\leq |X'X(f_1 - f_2)(x)|+ (e^{CM} -1) (X'X|f_1 - f_2| )(x)
\end{align*}
and
\begin{align*}
    |K f_1(x)-Kf_2(x)|\geq 	|X'X(f_1 - f_2)(x)|- (e^{CM}-1) (X'X|f_1 - f_2| )(x).
\end{align*}
Here $C$ is a positive constant depending on $\Omega$.	
Based on these, we can derive the $L^2$ norm estimate
\begin{equation}\label{EST:L2K}
\begin{split}
   &\hskip.5cm\|X'X(f_1 -f_2) \|_{L^2(\Omega_1)} -(e^{CM}-1)\|X'X|f_1-f_2|\|_{L^2(\Omega_1)}\\ 
   &  \leq \|Kf_1-Kf_2\|_{L^2(\Omega_1)} \leq e^{CM} \|X'X|f_1 - f_2| \|_{L^2(\Omega_1)}.
\end{split}
\end{equation}

Next we estimate $\|\nabla (Kf_1-Kf_2)\|_{L^2(\Omega_1)}$ by differentiating \eqref{ID:f1f2}:
\begin{align}\label{diff K}
	\nabla (K f_1-Kf_2)(x) 
	= \nabla (X'Xf_2 - X'Xf_1 )(x) + \nabla(X' R f_1- X'R f_2)(x).
\end{align}
It is sufficient to estimate the remainder term $\nabla X'Rf$. Note that for any $m\geq 2$,
\begin{align*}
    &\hskip.5cm   \frac{1}{m!}|\nabla (X'(Xf_1)^m-X'(Xf_2)^m)(x)|\\
    &=   \frac{1}{(m-1)!}|X'((Xf_1)^{m-1}\nabla_x Xf_1-(Xf_2)^{m-1}\nabla_x Xf_2)|\\
    &\leq   \frac{1}{(m-1)!}|X' (((Xf_1)^{m-1}-(Xf_2)^{m-1} )\nabla_x Xf_1)|+\frac{1}{(m-1)!}|X'((Xf_2)^{m-1}(\nabla_x Xf_1-\nabla_x Xf_2))|\\
   & \leq   \frac{1}{(m-2)!}(CM)^{m-1} X'X|f_1-f_2|+ \frac{1}{(m-1)!} (CM)^{m-1} X'X |\nabla (f_1-f_2)|.
\end{align*}
It follows that
\begin{align*}
|\nabla (X' R f_1- X'R f_2)(x)|  
&\leq \LC CM+ (CM)^2 +\frac{(CM)^3}{2!}+\cdots \RC X'X|f_1 - f_2| (x) \\
&\quad + \LC CM+ {(CM)^2\over 2!}+{(CM)^3\over 3!}+\cdots \RC X'X |\nabla (f_1-f_2)(x)|,
\end{align*}
which leads to
\begin{align*} 
\|\nabla X'R f_1- \nabla X'Rf_2\|_{L^2(\Omega_1)} \notag \leq CMe^{CM} \|X'X |f_1 -f_1| \|_{L^2(\Omega_1)} + (	e^{CM}-1) \|X'X |\nabla f_1-\nabla f_2|\|_{L^2(\Omega_1)}.
\end{align*}
Now since $f_1-f_2$ is compactly supported in $\overline\Omega$, we apply Proposition~\ref{prop:stability} to get
$$
 \|X'X|f_1-f_2|\|_{L^2(\Omega_1)}\leq  \|X'X|f_1-f_2|\|_{H^1(\Omega_1)}\leq  C \|f_1-f_2\|_{L^2(\R^n)}
$$
and similarly,
$$
\|X'X |\nabla f_1-\nabla f_2|\|_{L^2(\Omega_1)} \leq  C \| \nabla (f_1-f_2) \|_{L^2(\R^n)}.
$$
Thus, the above inequalities imply that
\begin{equation}\label{EST:gradiant XR}
\begin{split}
&\hskip.5cm\|\nabla X'R f_1- \nabla X'Rf_2\|_{L^2(\Omega_1)}  \\
&\leq CMe^{CM}\|f_1-f_2\|_{L^2(\R^n)} +C(e^{CM}-1) \| \nabla (f_1-f_2) \|_{L^2(\R^n)}= :\mathcal{F}.
\end{split}
\end{equation}
Note that Proposition~\ref{prop:stability} yields that
$$
C^{-1} \|f_1-f_2\|_{L^2(\R^n)}\leq \|X'Xf_2 - X'Xf_1\|_{H^1(\Omega_1)} \leq C \|f_1-f_2\|_{L^2(\R^n)}
$$
Combining with the second inequality of \eqref{EST:L2K} and also \eqref{diff K}, \eqref{EST:gradiant XR}, we now have
\begin{align*}
   &\hskip.5cm\|Kf_1- Kf_2\|_{H^1(\Omega_1)} \\
   &\leq e^{CM} \|X'X|f_1 - f_2| \|_{L^2(\Omega_1)}+\|\nabla (X'Xf_2 - X'Xf_1 )\|_{L^2(\Omega_1)} + \mathcal{F}\\
   &\leq C(1+ e^{CM}+Me^{CM})\|f_1-f_2\|_{L^2(\R^n)} +C(e^{CM}-1) \| \nabla (f_1-f_2) \|_{L^2(\R^n)}.
\end{align*}

On the other hand, combining with the first inequality of \eqref{EST:L2K} and also \eqref{diff K}, \eqref{EST:gradiant XR}, we obtain the lower bound
\begin{equation}\label{EST:Klower}
\begin{split}
	     &\hskip.5cm\|Kf_1- Kf_2\|_{H^1(\Omega_1)} \geq \|X'X(f_1 -f_2) \|_{H^1(\Omega_1)} -(e^{CM}-1)\|X'X|f_1-f_2|\|_{L^2(\Omega_1)} - \mathcal{F} \\ 
	     &\geq (C^{-1} -C(e^{CM}-1)-CMe^{CM})\|f_1-f_2\|_{L^2(\R^n)}-C(e^{CM}-1) \| \nabla (f_1-f_2) \|_{L^2(\R^n)}.
\end{split}\end{equation}
Note that when $M$ is decreasing to zero, so is $e^{CM}-1$. It implies that $C^{-1} -C(e^{CM}-1)-CMe^{CM}$ is nonnegative if $M$ is small enough.
Therefore, if we fix some constant $c>0$, given any $f_1, f_2$ satisfying $\|f_1-f_2\|_{L^2(\mathbb R^n)}\geq c\|\nabla(f_1-f_2)\|_{L^2(\mathbb R^n)}$, then we can derive from \eqref{EST:Klower} that
\begin{align}\label{EST:Kfinal}
\widetilde{C} \|f_1-f_2\|_{L^2(\mathbb R^n)}\leq \|Kf_1-Kf_2\|_{H^1(\Omega_1)},
\end{align}
where the positive constant $\widetilde{C} = C^{-1} -C(1+1/c)(e^{CM}-1)-CMe^{CM}$, depending on $c,M$ with $M$ sufficiently small. 
\end{proof}

\begin{remark}
	We note that the stability estimate improves when the magnitude of $f$ becomes smaller. To explain this, we observe from \eqref{EST:Kfinal} that when $M$ is decreasing, the term $C^{-1}\|f_1-f_2\|_{L^2(\R^n)}$ on the right-hand side will dominate. Hence, the whole estimate becomes slightly stabler since the coefficient $\widetilde{C}$ is increasing. 
\end{remark}

Finally we make a comment on the connection between the stability estimate \eqref{EST:g} in Corollary~\ref{cor:stability} and the lower bound \eqref{EST:Klower}.
Indeed \eqref{EST:Klower} implies \eqref{EST:g} when either one of $f_j$ is zero.

More precisely, we take $f_1=\tau g\in \mathcal{M}_1$ with $\tau,\,g\geq 0$, and $f_2\equiv 0$ with $\widetilde{M}:=\|\tau g\|_{C^1(\R^n)}$. The estimate
\eqref{EST:Klower} yields that
\begin{align*}
\|K[\tau g]-K[0]\|_{H^1(\Omega_1)} \geq
(C^{-1} -C(e^{C\widetilde{M}}-1)-C\widetilde{M}e^{C\widetilde{M}})\|\tau g\|_{L^2(\R^n)}-C(e^{C\widetilde{M}}-1) \| \nabla (\tau g) \|_{L^2(\R^n)}.
\end{align*}
Dividing by $\tau$ and letting $\tau \rightarrow 0$ (then $\widetilde{M}\rightarrow 0$), one has
$$
\|\p_\tau|_{\tau=0} K[\tau g]\|_{H^1(\Omega_1)} \geq C^{-1} \|g \|_{L^2(\R^n)}.
$$

\subsection{Single pixel transform on Riemannian manifolds}

As mentioned in the introduction, the second author \cite{KDMV16} proved the uniqueness of the single pixel X-ray transform $K$ in the Euclidean space $\mathbb R^n$. In this section, we show that the proof for the Euclidean case can be generalized to the case of non-trivial geometries. 

Let $(M,g)$ be an $n$-dimensional, $n\geq 2$, compact non-trapping Riemannian manifold with smooth strictly convex boundary $\p M$. Here non-trapping means that every geodesic exits the manifold in both directions in finite times. Let $SM$ be the unit sphere bundle consisting of all unit vectors on $M$, so any $(x,v)\in SM$ satisfying $\|v\|_{g(x)}=1$. Given any $(x,v)\in SM$, let $\gamma_{x,v}$ be the geodesic with initial conditions $\gamma_{x,v}(0)=x$, $\dot\gamma_{x,v}(0)=v$. We define the X-ray transform of a function $f$ on $(M,g)$ as
$$X f(x,v)=\int f(\gamma_{x,v}(t))\, dt, \quad (x,v)\in SM.$$
Since $(M,g)$ is non-trapping, the above integral is indeed over a finite interval. Moreover, $Xf(x,v)=Xf(\gamma_{x,v}(t),\dot\gamma_{x,v}(t))$ for all $t$. Then the single pixel X-ray transform on $(M,g)$ is defined by
$$K f (x)=\int_{S_xM} e^{-Xf(x,v)} \,dv,\quad x\in M.$$
Let $C(M)$ be the space of continuous functions on $M$.
\begin{theorem}
	Let $(M,g)$ be a compact non-trapping Riemannian manifold with smooth strictly convex boundary. Assume that the X-ray transform $X$ is injective on $C(M)$, then the single pixel X-ray transform $K$ is injective on $C(M)$, in other words, $Kf=Kh$ implies that $f=g$ for $f,g\in C(M)$.
\end{theorem}
\begin{proof}
	Consider 
	\begin{align*}
     \LA Kf-Kh, f-h\RA= & \int_M (Kf-Kh)(x)(f-h)(x)\, dx\\
     = & \int_{SM} (e^{-Xf(x,v)}-e^{-Xh(x,v)}) (f-h)(x)\, dxdv\\
     = & \int_{\p_+SM} \, d\mu(x,v)\int (e^{-Xf(\gamma_{x,v}(t),\dot\gamma_{x,v}(t))}-e^{-Xh(\gamma_{x,v}(t),\dot\gamma_{x,v}(t))}) (f-h)(\gamma_{x,v}(t))\, dt.
	\end{align*}
	The last equality is a direct application of the Santalo's formula, see e.g. \cite[Lemma 3.3.2]{SharafutdinovNote}. Here $\p_+SM$ is the set of all unit inward pointing vectors on the boundary $\p M$, $d\mu$ is a measure on $\p_+SM$ which vanishes in directions tangent to $\p M$.
	
	Notice that $Xf(\gamma_{x,v}(t),\dot\gamma_{x,v}(t))$ is invariant w.r.t. $t$, thus
	\begin{align*}
	\LA Kf-Kh, f-h\RA= & \int_{\p_+SM}(e^{-Xf(x,v)}-e^{-Xh(x,v)}) \, d\mu(x,v)\int (f-h)(\gamma_{x,v}(t))\, dt\\
	=&  \int_{\p_+SM}(e^{-Xf(x,v)}-e^{-Xh(x,v)}) (Xf(x,v)-Xh(x,v))\, d\mu(x,v).
	\end{align*}
	The integrand in the last integral has the form $(e^{-u}-e^{-v})(u-v)$, which is non-positive, and it equals zero if and only if $u=v$. Therefore, if $Kf=Kh$, we have that $Xf(x,v)=Xh(x,v)$ for all $(x,v)\in\p_+SM$. By our assumption, the X-ray transform $X$ is injective, thus $f=h$.
\end{proof}

It is known that the X-ray transform is injective on {\it simple} manifolds \cite{Mu77, MuRo78}, which are compact non-trapping manifolds with strictly convex boundary and free of conjugate points. In dimension $\geq 3$, $X$ is injective on compact non-trapping manifolds with strictly convex boundary, which admit convex foliations \cite{UV2016}. The convex foliation condition allows the existence of conjugate points.

In the current paper, we only consider the uniqueness of $K$ on Riemannian manifolds. It's reasonable to expect that stability estimates similar to Theorem \ref{thm:stability} will hold on simple manifolds as well. In particular, stability estimates of the normal operator $X'X$ on simple manifolds can be found in \cite{SU2004}.

\section{Numerical experiments}
In this section, we conduct several numerical experiments to corroborate our theoretical results above. We use the Shepp-Logan phantom with  $101\times 101$ pixels for illustrations.
 
Assume $g$ is the measured data. More precisely, $g(x)$ is the single pixel X-ray transform of $f$ at the point $x$. For all the numerical tests below, $g$ has the same resolution as $f$.
To reconstruct $f$ from the data $g$, we use Gauss-Newton method to minimize the following functional:
\[
\arg\min_f\|Kf-g\|_{L^2}.
\]
To compute the gradient of the above functional, note that for
\[
Kf(x)=\int_{\mathbb{S}^{n-1}}e^{-Xf(x,\theta)}\mathrm{d}\theta,
\]
one can see that the Frech\'et derivative of $K$ at $f$ is given as
\[
K'[f](h)(x)=\int_{\mathbb{S}^{n-1}}e^{-Xf(x,\theta)}(-Xh(x,\theta))\mathrm{d}\theta
\]
for any function $h$. For computation of the X-ray transform $Xf$, we adapt the code provided in Carsten H{\o}ilund's lecture notes \cite{hoilund2007radon}.\\

We first reconstruct $f$ when the data $g$ is not noised. The results are shown in Figure \ref{N1}. The true image of Shepp-Logan is in the middle. The image of $Kf$ is on the left. %We can see that the $Kf$ and $f$ are already quite similar.
\begin{figure}[ht]
\centering
\includegraphics[height=4.1cm]{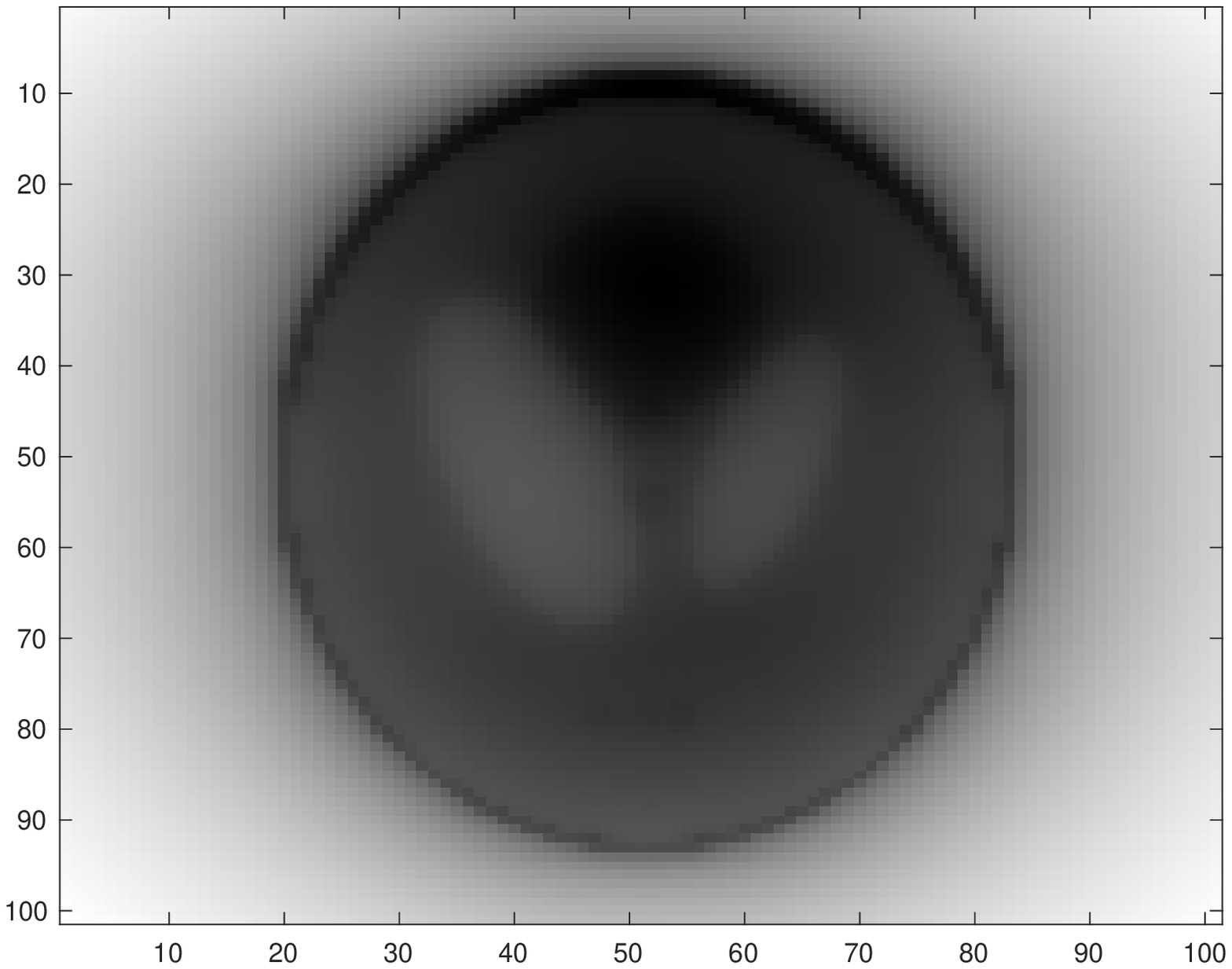}~
\includegraphics[height=4.1cm]{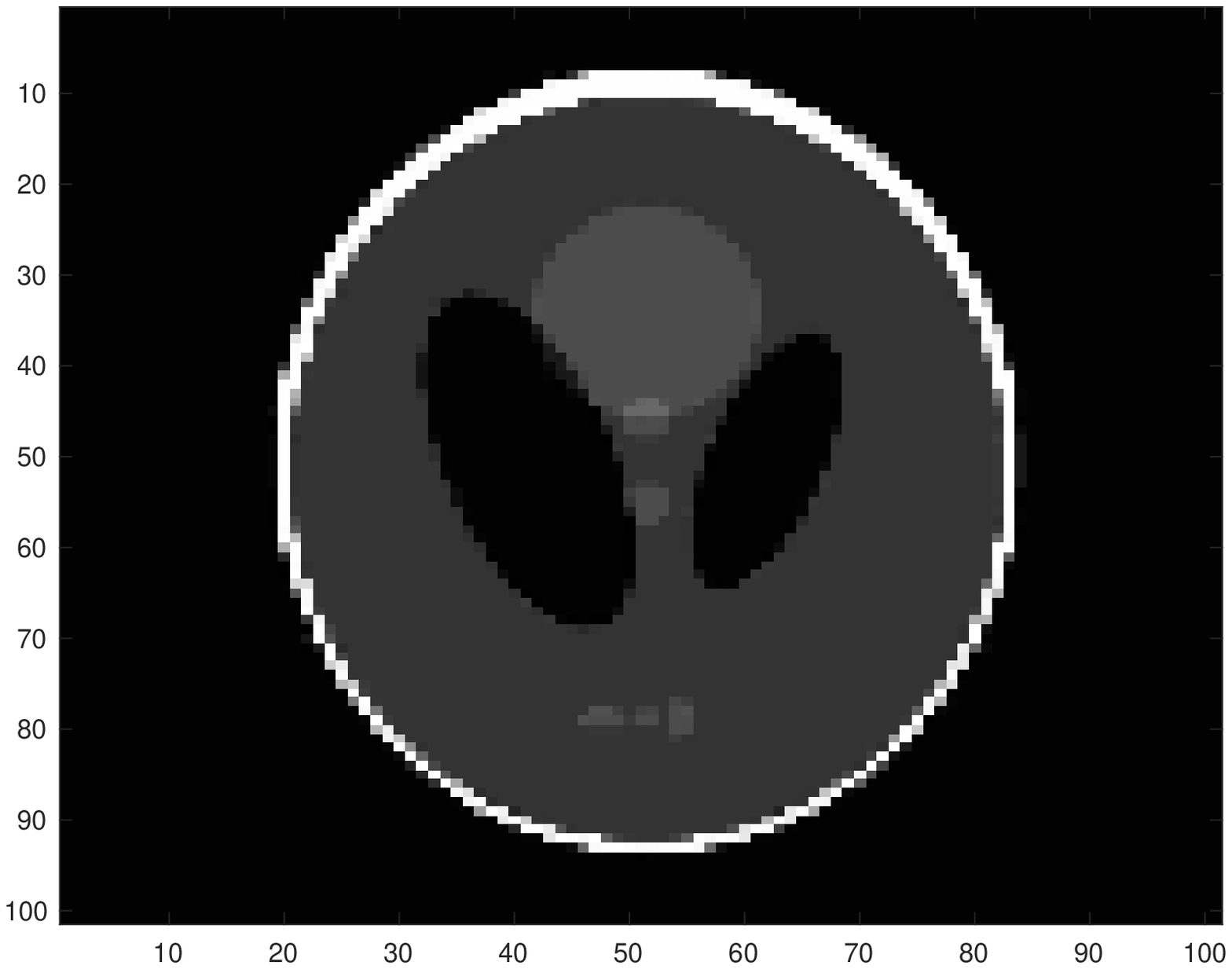}~
\includegraphics[height=4.1cm]{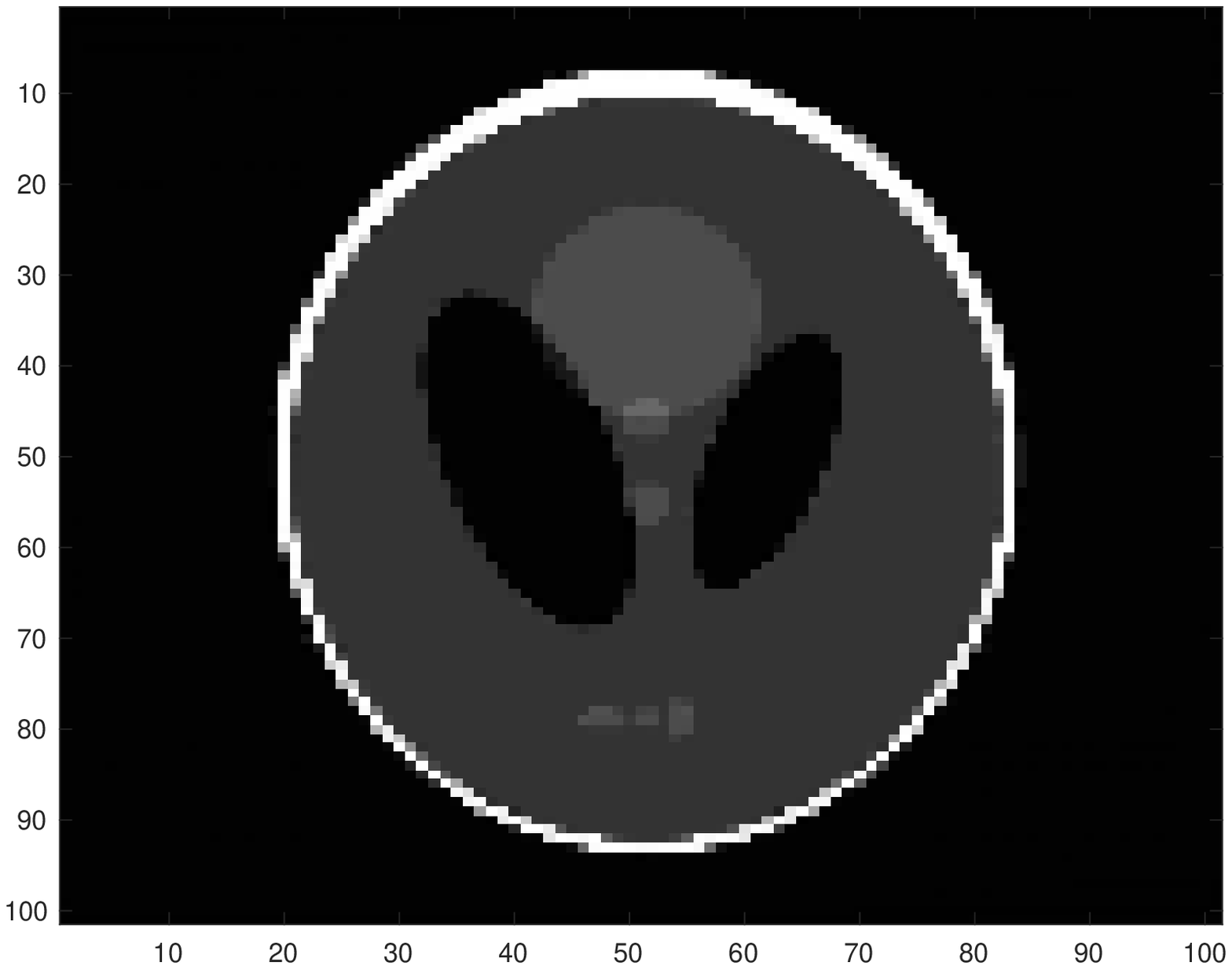}
\caption{left: true data; middle: true image, right: reconstructed image}\label{N1}
\end{figure}

Next we reconstruct $f$ with $g$ polluted by Gaussian random noises. The noise level is compared with $Kf-2\pi$, not $Kf$ itself, since when $f=0$, $Kf=2\pi$ in $\R^2$. Notice that adding noise directly to $Kf$ will lead to a complete loss of information for small $f$'s. The results are presented in Figure \ref{E2}, where the reported error is measured in $L^2$ norm. One can see that the magnitude of error is almost linearly dependent on the level of noise. This confirms the Lipschitz stability result in Theorem \ref{thm:stability}.
\begin{figure}[h]
\centering
%    \begin{subfigure}[t]{0.32\textwidth}
%      \includegraphics[width=\textwidth]{mag2data100noise0001minus1}
%      \caption{the data with $0.1\%$ relative noise}
%    \end{subfigure}
    \begin{subfigure}[t]{0.32\textwidth}
      \includegraphics[width=\textwidth]{mag2recon100noise0minus1.eps}
      \caption{the reconstruction with no noise}
    \end{subfigure}
    \begin{subfigure}[t]{0.32\textwidth}
      \includegraphics[width=\textwidth]{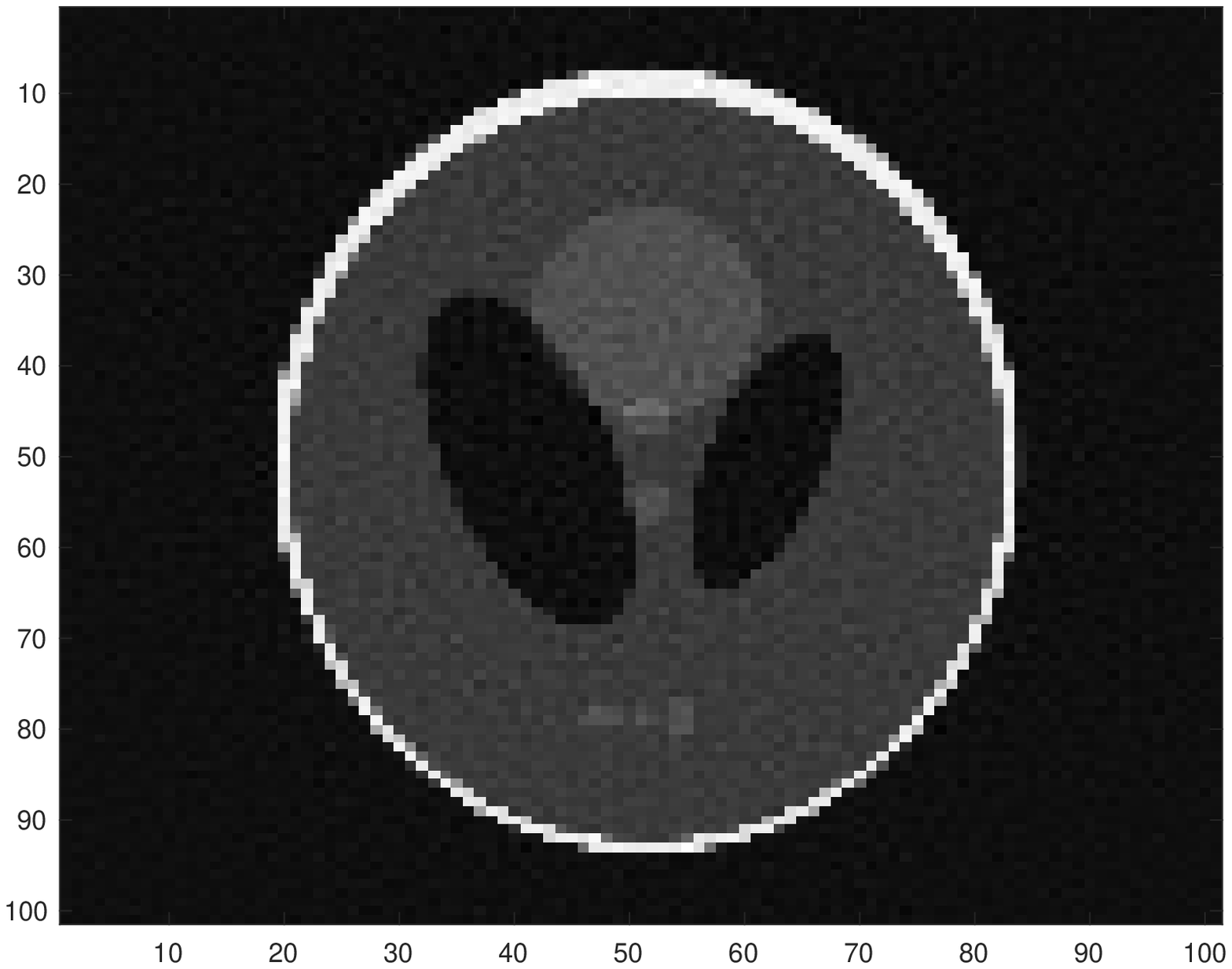}
      \caption{the reconstruction with $0.1\%$ relative noise; the error is $0.7289$}
    \end{subfigure}\\
    
%       \begin{subfigure}[t]{0.32\textwidth}
%      \includegraphics[width=\textwidth]{mag2data100noise0005minus1}
%      \caption{the data with $0.5\%$ relative noise}
%    \end{subfigure}
    \begin{subfigure}[t]{0.32\textwidth}
      \includegraphics[width=\textwidth]{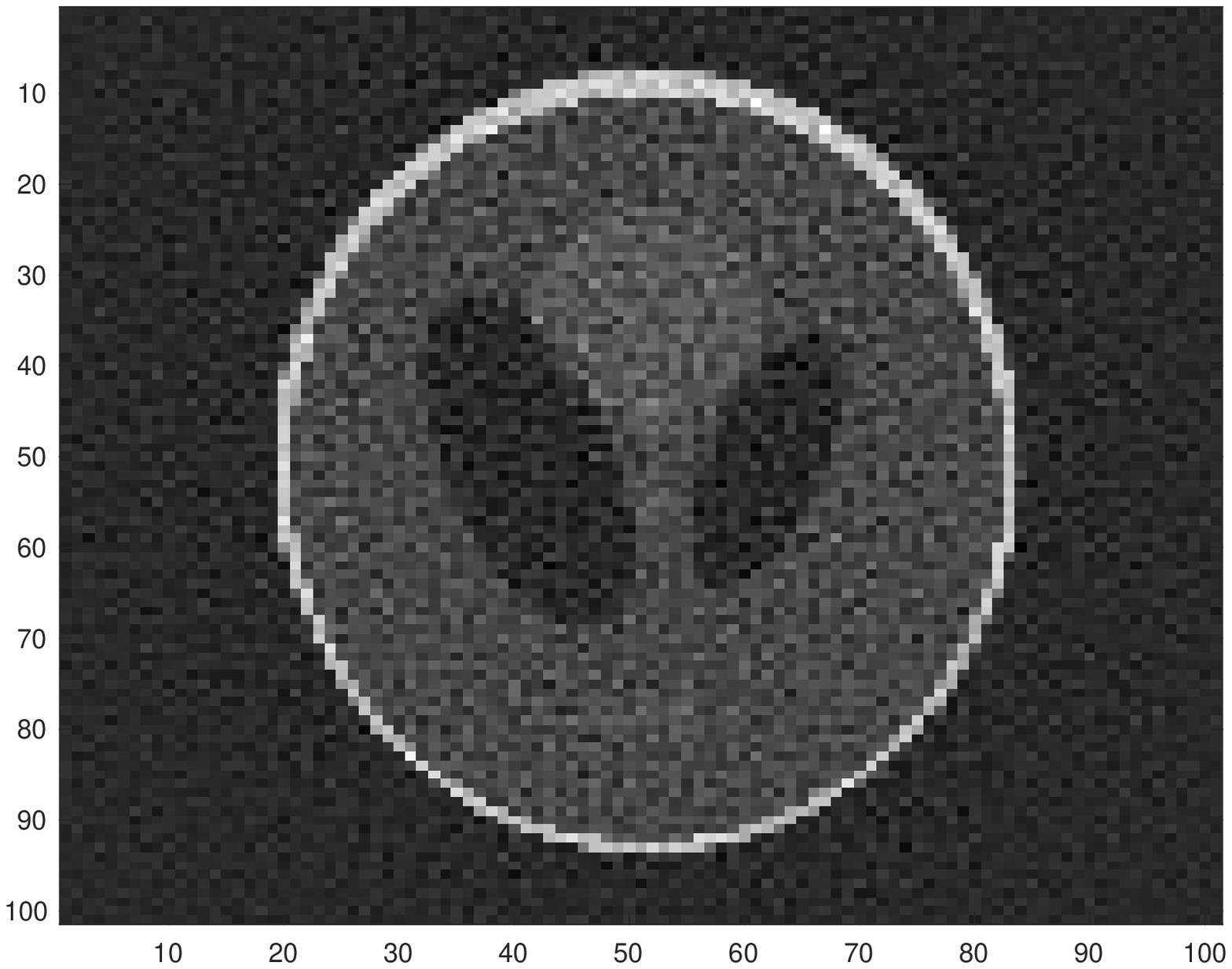}
      \caption{the reconstruction with $0.5\%$ relative noise; the error is $3.5906$}
    \end{subfigure}
%    
%          \begin{subfigure}[t]{0.32\textwidth}
%      \includegraphics[width=\textwidth]{mag2data100noise001minus1}
%      \caption{the data with $1\%$ relative noise}
%    \end{subfigure}
    \begin{subfigure}[t]{0.32\textwidth}
      \includegraphics[width=\textwidth]{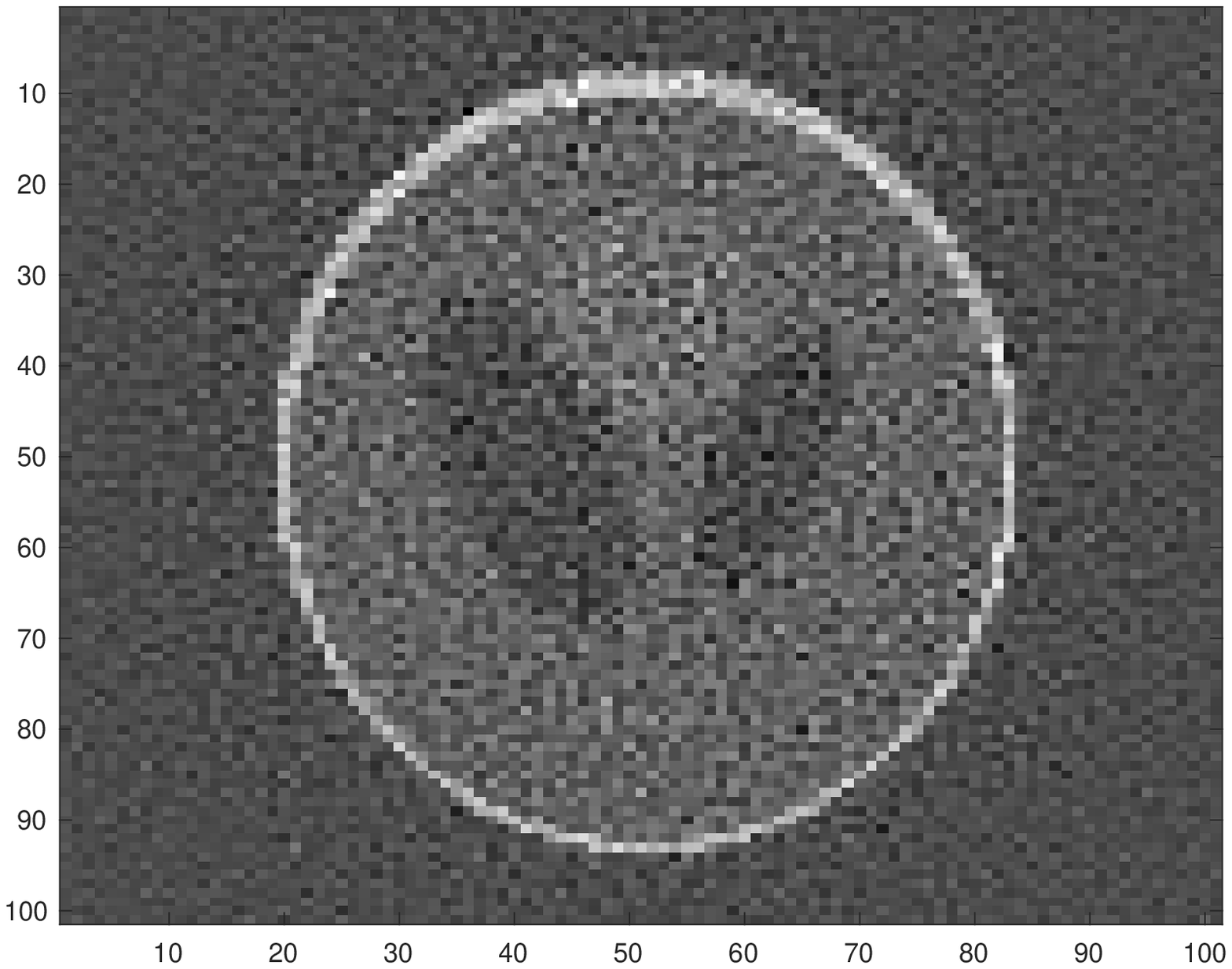}
      \caption{the reconstruction with $1\%$ relative noise; the error is $7.1819$}
    \end{subfigure}

    \caption{Recovery with different noises}\label{E2}
  \end{figure}

Finally, we reconstruct the same Shepp-Logan phantom with different magnitudes of $f$. Although we do not manually add noises, computation itself generates noises. The results are displayed in Figures~\ref{EWN}. One can see that for both $f$ and $10f$, the reconstructions perform quite well, %are almost the same good, 
while the algorithm takes more steps to converge for the case $10f$. Moreover, Figure~\ref{EWN} shows that the quality of image deteriorates if the magnitude of $f$ becomes even larger. For $40f$, the reconstruction is already a total failure. This result suggests that the Lipschitz stability derived for small $f$'s no longer holds for large ones.

%  \begin{table}[h]
%  \centering
%  \begin{tabular}{|c |c | c |}
%  \hline
% magnitude & error& error/magnitude\\
%   \hline
% 0.02 &0.1280& 6.4000\\
%   \hline
% 1 &7.1819&7.1819\\
%   \hline
% 5 &66.6106& 13.3221\\
%  \hline
%  \end{tabular}
%  \caption{Parameters for the algorithm}\label{Error}
%\end{table}
  
% \RuYu{In fig 3, why at step 0, the error is 30?} 
   \begin{figure}[h]
  \centering
%          \begin{subfigure}[t]{0.4\textwidth}
%      \includegraphics[width=\textwidth]{mag004data100noise0minus1}
%      \caption{the data for $0.02f$ with no noise}
%    \end{subfigure}
%    \begin{subfigure}[t]{0.4\textwidth}
%      \includegraphics[width=\textwidth]{mag004recon100noise0minus1}
%      \caption{reconstruction for $0.02f$ with no noise}
  %  \end{subfigure}
%    \begin{subfigure}[t]{0.32\textwidth}
%      \includegraphics[width=\textwidth]{mag2data100noise0minus1_wb}
%      \caption{the data for $f$ with no noise}
%    \end{subfigure}
    \begin{subfigure}[t]{0.32\textwidth}
      \includegraphics[width=\textwidth]{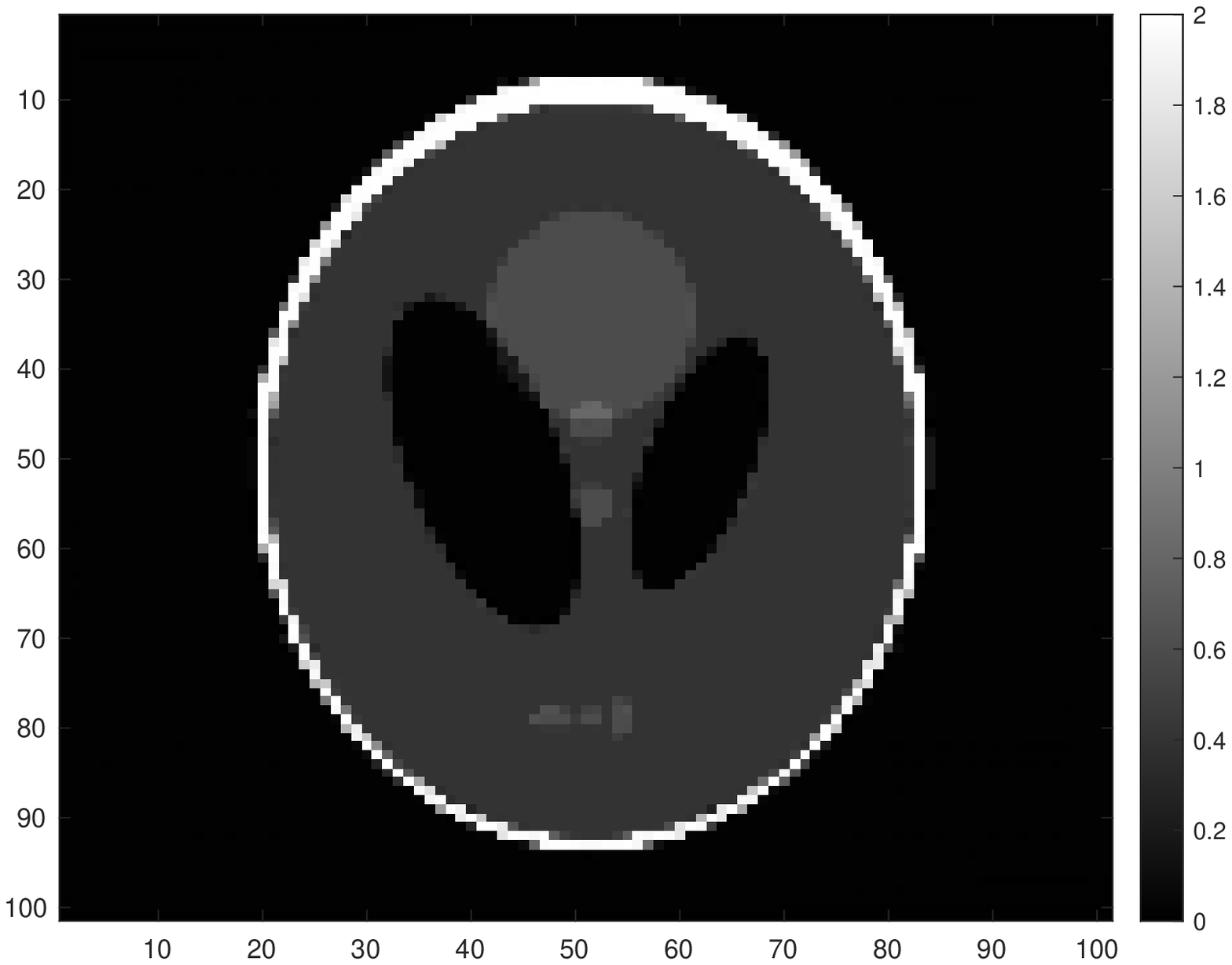}
      \caption{reconstruction for $f$ with no noise}
      \end{subfigure}
       \begin{subfigure}[t]{0.32\textwidth}
      \includegraphics[width=\textwidth]{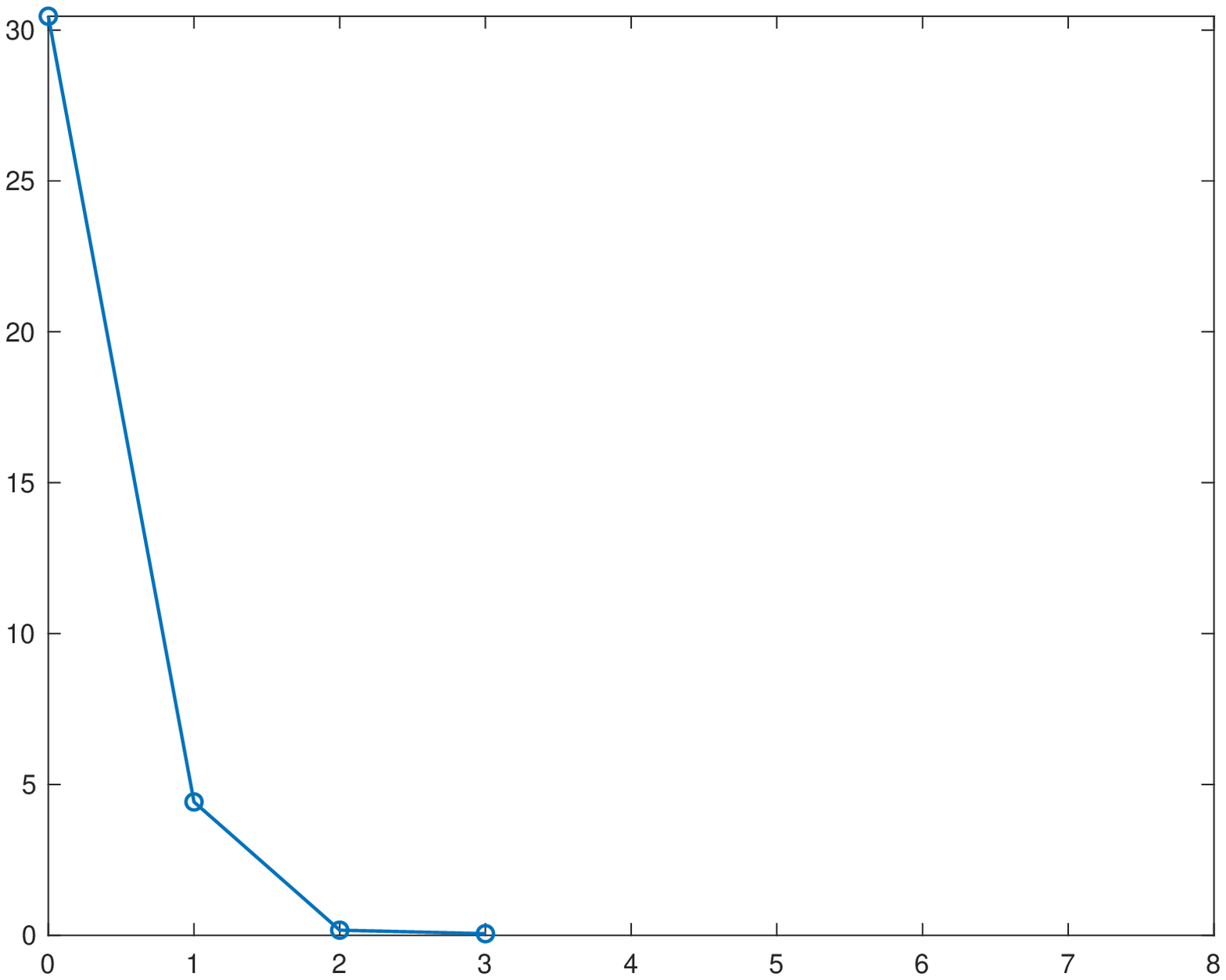}
      \caption{errors versus steps}
      \end{subfigure}
      
%          \begin{subfigure}[t]{0.32\textwidth}
%      \includegraphics[width=\textwidth]{mag20data100noise0minus1_wb}
%      \caption{the data for $10f$ with no noise}
%    \end{subfigure}
    \begin{subfigure}[t]{0.32\textwidth}
      \includegraphics[width=\textwidth]{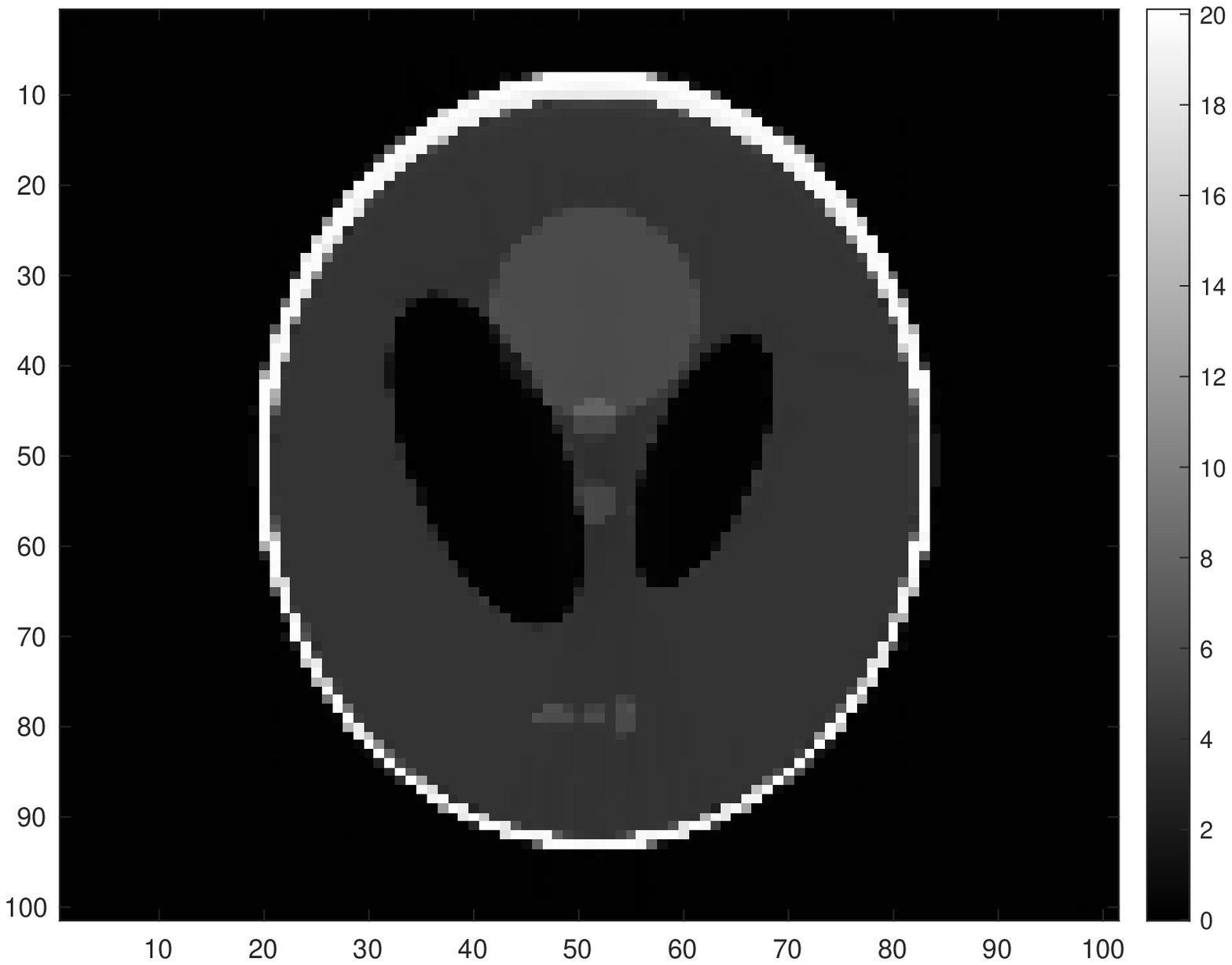}
      \caption{reconstruction for $10f$ with no noise}
       \end{subfigure}
         \begin{subfigure}[t]{0.32\textwidth}
      \includegraphics[width=\textwidth]{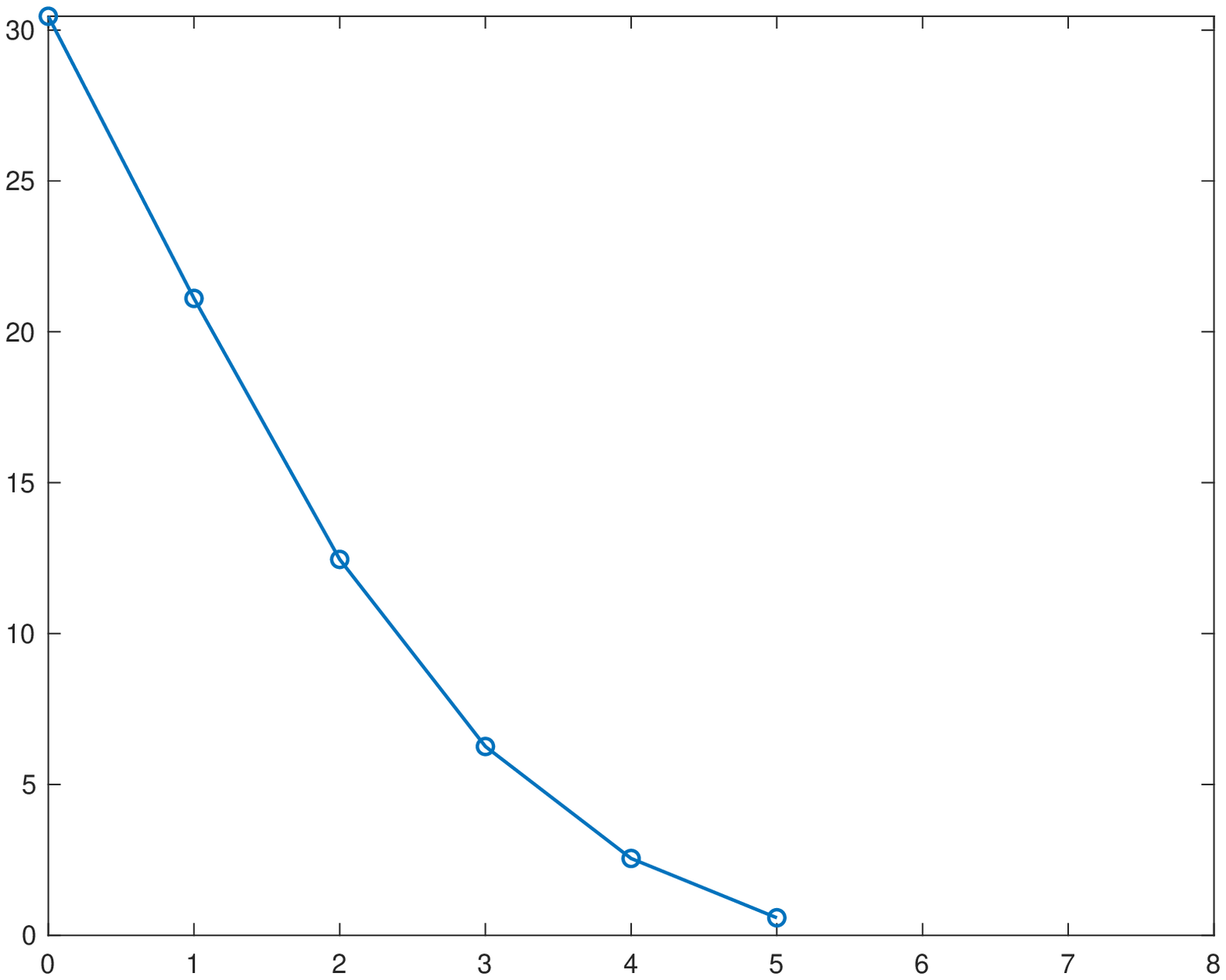}
      \caption{errors/10 versus steps}
      \end{subfigure}
      
%              \begin{subfigure}[t]{0.32\textwidth}
%      \includegraphics[width=\textwidth]{mag40data100noise0minus1_wb}
%      \caption{the data for $20f$ with no noise}
%    \end{subfigure}
    \begin{subfigure}[t]{0.32\textwidth}
      \includegraphics[width=\textwidth]{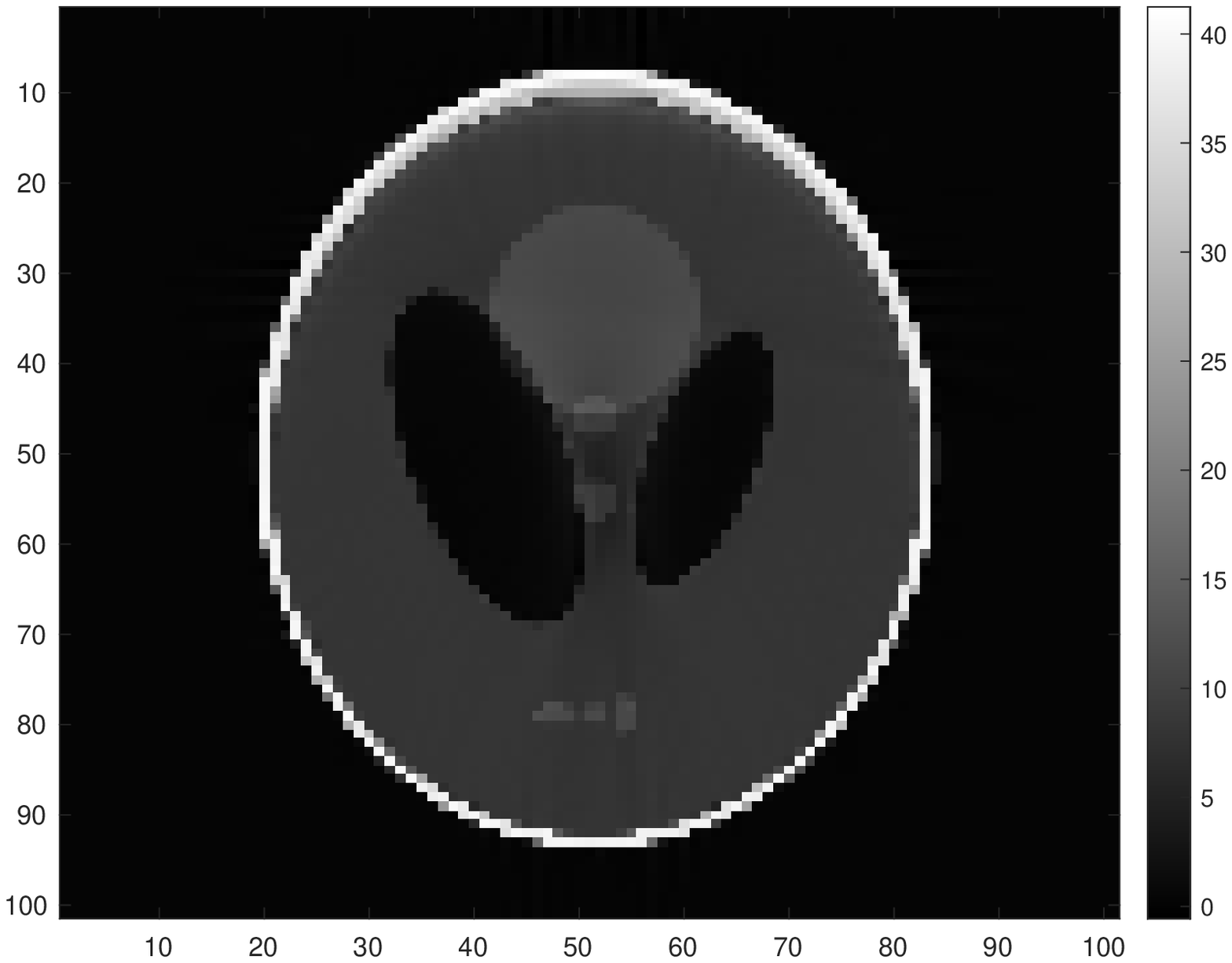}
      \caption{reconstruction for $20f$ with no noise}
    \end{subfigure}
      \begin{subfigure}[t]{0.32\textwidth}
      \includegraphics[width=\textwidth]{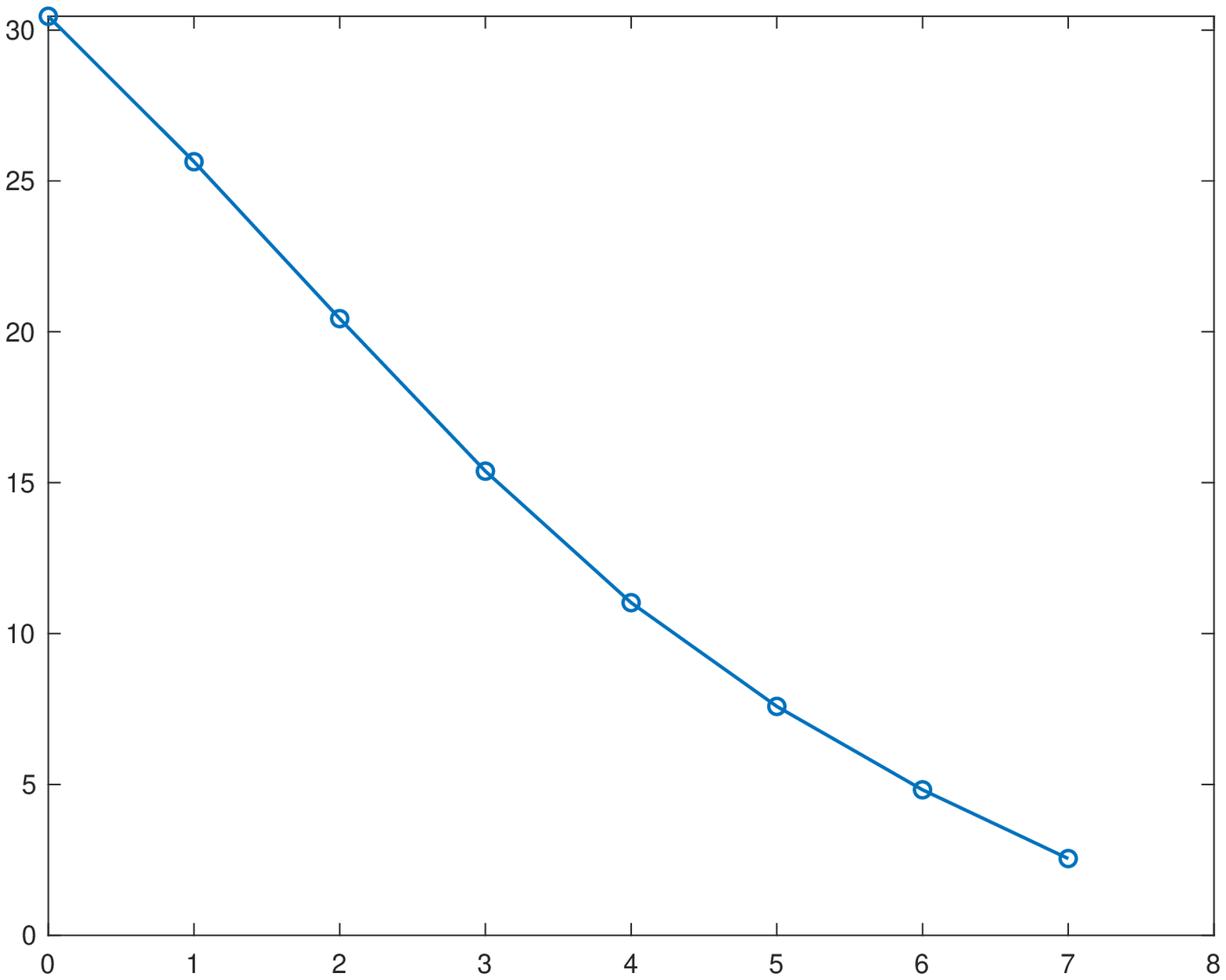}
      \caption{errors/20 versus steps}
      \end{subfigure}
    
%            \begin{subfigure}[t]{0.32\textwidth}
%      \includegraphics[width=\textwidth]{mag80data100noise0minus1_wb}
%      \caption{the data for $40f$ with no noise}
%    \end{subfigure}
    \begin{subfigure}[t]{0.32\textwidth}
      \includegraphics[width=\textwidth]{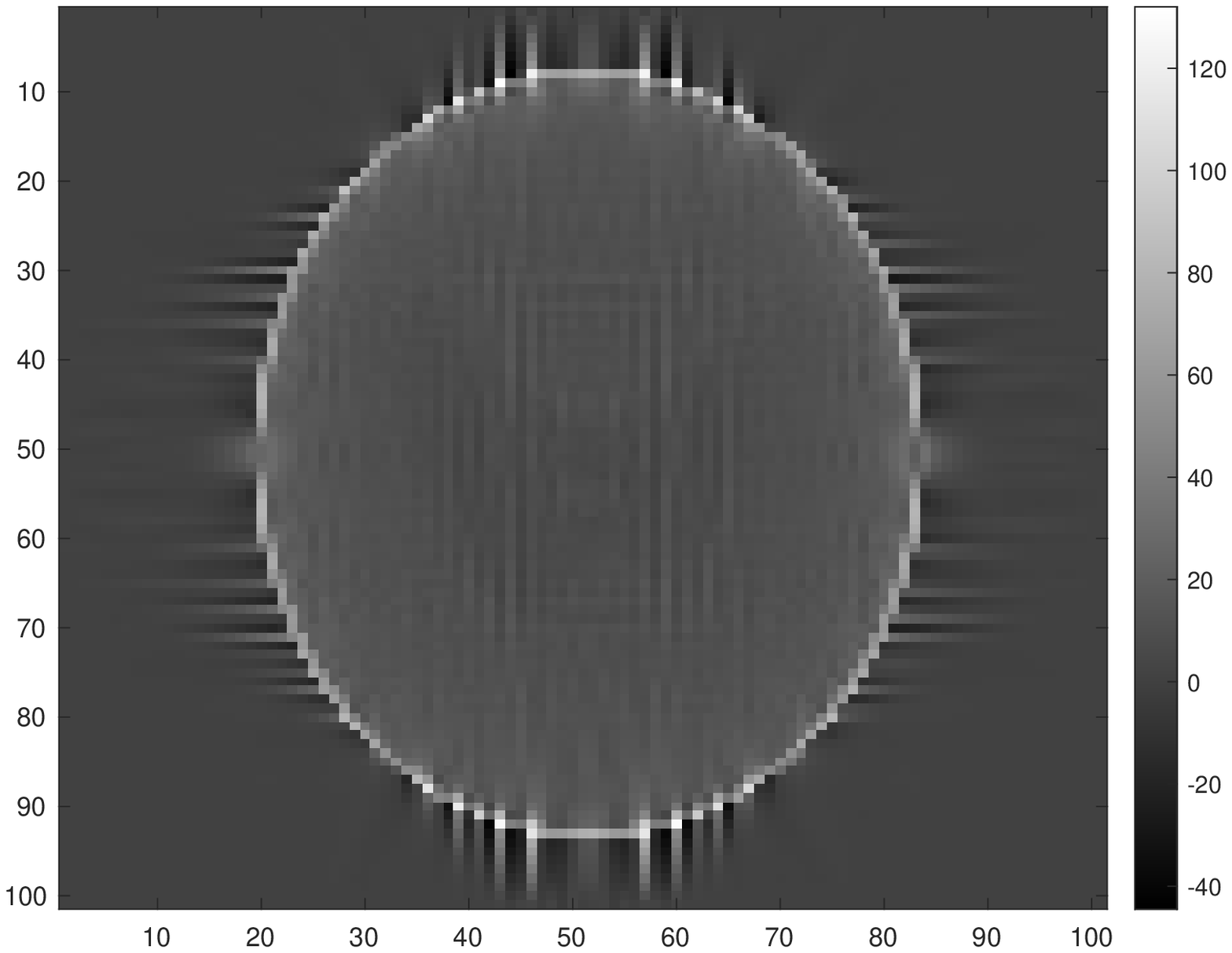}
      \caption{reconstruction for $40f$ with no noise}
      \end{subfigure}
        \begin{subfigure}[t]{0.32\textwidth}
      \includegraphics[width=\textwidth]{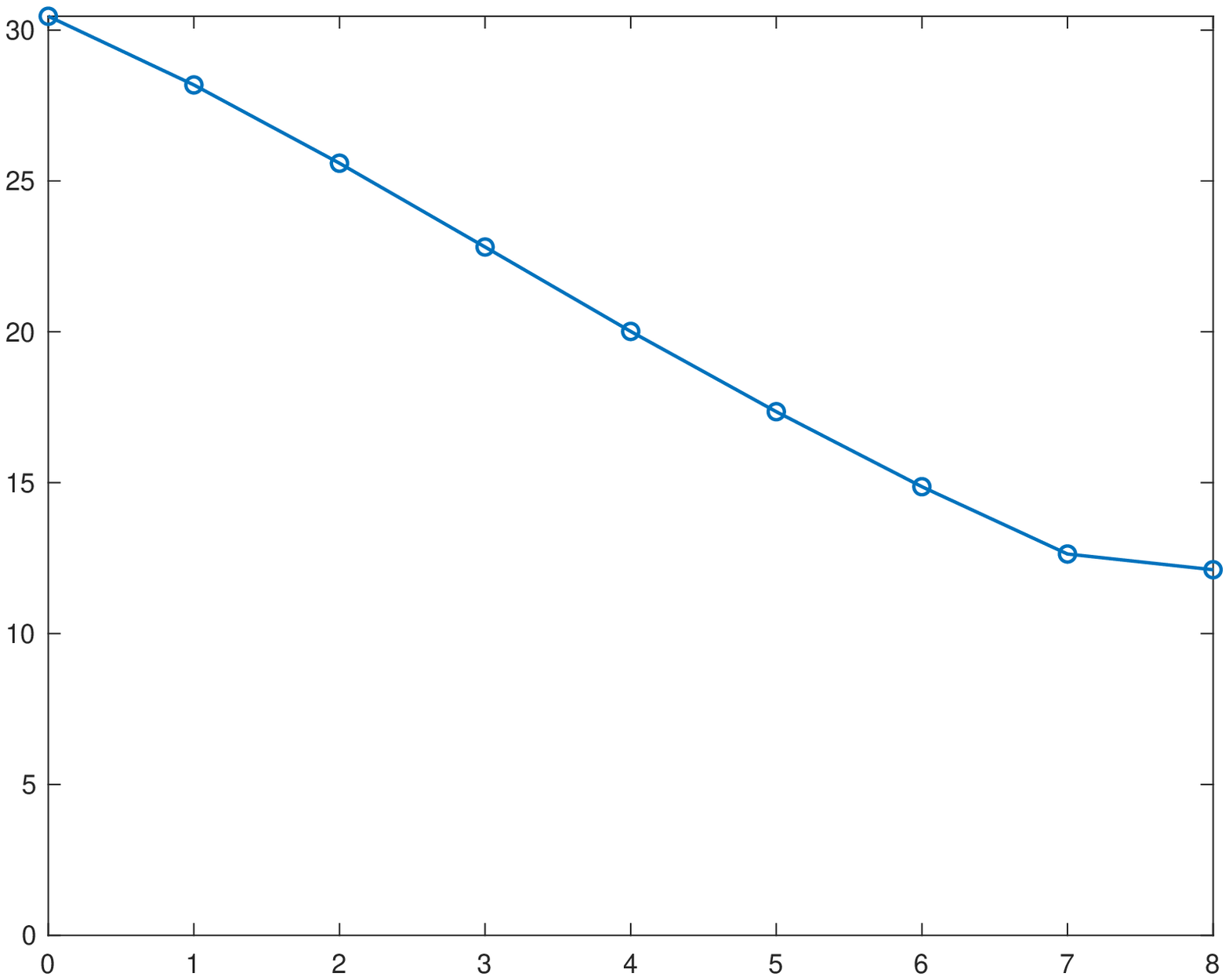}
      \caption{errors/40 versus steps}
      \end{subfigure}

    \caption{Reconstruction without noise}\label{EWN}
  \end{figure} 
%  \begin{figure}[h]
%  \centering
%          \begin{subfigure}[t]{0.32\textwidth}
%      \includegraphics[width=\textwidth]{mag004data100noise001minus1}
%      \caption{the data for $0.02f$ with $1\%$ relative noise}
%    \end{subfigure}
%    \begin{subfigure}[t]{0.32\textwidth}
%      \includegraphics[width=\textwidth]{mag004recon100noise001minus1}
%      \caption{reconstruction for $0.02f$ with $1\%$ relative noise}
%    \end{subfigure}\\
%    \begin{subfigure}[t]{0.32\textwidth}
%      \includegraphics[width=\textwidth]{mag2data100noise001minus1}
%      \caption{the data for $f$ with $1\%$ relative noise}
%    \end{subfigure}
%    \begin{subfigure}[t]{0.32\textwidth}
%      \includegraphics[width=\textwidth]{mag2recon100noise001minus1}
%      \caption{reconstruction for $f$ with $1\%$ relative noise}
%      \end{subfigure}\\
%          \begin{subfigure}[t]{0.32\textwidth}
%      \includegraphics[width=\textwidth]{mag10data100noise001minus1}
%      \caption{the data for $5f$ with $1\%$ relative noise}
%    \end{subfigure}
%    \begin{subfigure}[t]{0.32\textwidth}
%      \includegraphics[width=\textwidth]{mag10recon100noise001minus1}
%      \caption{reconstruction for $5f$ with $1\%$ relative noise}
%    \end{subfigure}
%    
%  
%
%    \caption{Reconstruction with noises}\label{EE}
%  \end{figure} 

\section*{Acknowledgments}
RL is partially supported by the National Science Foundation (NSF) through grant DMS-2006731. GU is partly supported by NSF, a Simons Fellowship, a Walker Family Professorship at UW, and a Si Yuan Professorship at IAS, HKUST. Part of this research was performed while GU was visiting the Institute for Pure and Applied Mathematics (IPAM) in Fall 2021. HZ is partly supported by NSF grant DMS-2109116.

\bibliographystyle{plain}
\bibliography{Xray}

\begin{thebibliography}{10}

\bibitem{Helgason}
S.~Helgason.
\newblock {\em Integral Geometry and Radon Transforms}.
\newblock Springer, 2010.

\bibitem{hoilund2007radon}
C.~H{\o}ilund.
\newblock The radon transform.
\newblock {\em Aalborg University}, 12, 2007.

\bibitem{IM2019}
J.~Ilmavirta and F.~Monard.
\newblock Integral geometry on manifolds with boundary and applications.
\newblock {\em The Radon Transform: the first 100 years and beyond}, pages
  43--114, 2019.

\bibitem{Isakov93}
V.~Isakov.
\newblock On uniqueness in inverse problems for semilinear parabolic equations.
\newblock {\em Archive for Rational Mechanics and Analysis}, 124(1):1--12,
  1993.

\bibitem{Mu77}
R.~G. Mukhometov.
\newblock The reconstruction problem of a two-dimensional {R}iemannian metric,
  and integral geometry.
\newblock {\em Dokl. Akad. Nauk SSSR (in Russian)}, 232:32--35, 1977.

\bibitem{MuRo78}
R.~G. Mukhometov and V.~G. Romanov.
\newblock On the problem of finding an isotropic {R}iemannian metric in an
  n-dimensional space.
\newblock {\em Dokl. Akad. Nauk SSSR (in Russian)}, 243:41--44, 1978.

\bibitem{Natterer2001book}
F.~Natterer.
\newblock {\em The Mathematics of Computerized Tomography}.
\newblock SIAM, 2001.

\bibitem{PSU2014}
G.P. Paternain, M.~Salo, and G.~Uhlmann.
\newblock Tensor tomography: progress and challenges.
\newblock {\em Chinese Annals of Math. Ser. B}, 35:399--428, 2014.

\bibitem{KDMV16}
R.~R.~Macdonald R.~S.~Kemp, A.~Danagoulian and J.~R. Vavrek.
\newblock Physical cryptographic verification of nuclear warheads.
\newblock {\em PNAS}, 113 (31):8618--8623, 2016.

\bibitem{Radon1917}
J.~Radon.
\newblock \"uber die bestimmung von funktionen durch ihre integralwerte lngs
  gewisser mannigfaltigkeiten.
\newblock {\em Berichte \"uber die Verhandlungen der K\"oniglich-S\"achsischen
  Akademie der Wissenschaften zu Leipzig, Mathematisch-Physische Klasse},
  69:262--277, 1917.

\bibitem{SharafutdinovNote}
V.~Sharafutdinov.
\newblock Ray transform on {R}iemannian manifolds. {E}ight lectures on integral
  geometry.
\newblock {\em http://math.nsc.ru/ sharafutdinov/files/Lectures.pdf}.

\bibitem{SU2004}
P.~Stefanov and G.~Uhlmann.
\newblock Stability estimates for the {X}-ray transform of tensor fields and
  boundary rigidity.
\newblock {\em Duke Math. J.}, 123:445--467, 2004.

\bibitem{stefanov2009linearizing}
P.~Stefanov and G.~Uhlmann.
\newblock Linearizing non-linear inverse problems and an application to inverse
  backscattering.
\newblock {\em Journal of Functional Analysis}, 256(9):2842--2866, 2009.

\bibitem{StefanovUhlmannBook}
P.~Stefanov and G.~Uhlmann.
\newblock {\em Microlocal Analysis and Integral Geometry (working title)}.
\newblock draft version, 2018.

\bibitem{Taylor1981}
M.~Taylor.
\newblock {\em Pseudodifferential Operators}, volume 34, Princeton Mathematics
  Series.
\newblock Princeton, NJ: Princeton University Press, 1981.

\bibitem{UV2016}
G.~Uhlmann and A.~Vasy.
\newblock The inverse problem for the local geodesic ray transform.
\newblock {\em Inventiones Mathematicae}, 205:83--120, 2016.

\end{thebibliography}

\end{document}